\documentclass{article}
\usepackage{amsmath}
\usepackage{amsfonts}
\usepackage{amsthm}
\usepackage{latexsym}
\usepackage{amssymb}
\usepackage{verbatim}
\usepackage{enumerate}
\font\logic=msam10 at 10pt

\newcommand{\restrict}{\mbox{\logic\char'026}}

\newcommand{\cP}{\mathcal{P}}

\newcommand{\cX}{\mathcal{X}}
\newcommand{\cL}{\mathcal{L}}

\newtheorem{thrm}{Theorem}[section]

\newtheorem{prop}[thrm]{Proposition}

\newtheoremstyle{hdefinition}%
  {\topsep}%
  {\topsep}%
  {\upshape}
  {}%
  {\bfseries}%
  {.}
  { }%
  {\thmnumber{#2 }\thmname{#1}\thmnote{ \rm(#3)}}%

\newtheoremstyle{hclaim}%
  {\topsep}%
  {\topsep}%
  {\itshape}%
  {}%
  {\bfseries}%
  {.}
  { }%
  {\thmname{#1}\thmnote{ \rm#3}}%

\newtheoremstyle{hnotation}%
  {\topsep}%
  {\topsep}%
  {\upshape}%
  {}%
  {\bfseries}%
  {.}
  { }%
  {\thmname{#1}\thmnote{ \rm#3}}%

\theoremstyle{hclaim}
\newtheorem*{claim*}{Claim}

\theoremstyle{hdefinition}
\newtheorem{df}[thrm]{Definition}

\newtheorem{ques}[thrm]{Question}

\theoremstyle{hclaim}

\theoremstyle{hnotation}

\newcommand{\cA}{\mathcal{A}}
\newcommand{\cF}{\mathcal{F}}

\begin{document}

\title{An extendible structure with a rigid elementary extension}


\author{Paul B. Larson\thanks{Supported in part by NSF Grant
  DMS-1201494.}\\
 Miami University\\
Oxford, Ohio USA\\
\and Saharon Shelah\thanks{Research partially supported by NSF grant no: 1101597, and by the
European Research Council grant 338821. Paper No. 1130 on Shelah's list. }\\
Hebrew University of Jerusalem\\Rutgers University}

\maketitle

A countable structure is said to be \emph{extendible} if it is $\cL_{\infty, \aleph_{0}}$-elementarily equivalent to an uncountable structure.
A theorem of Su Gao says that a countable structure is extendible if and only if its automorphism group is not closed in the space of injections from the domain of the structure to itself. In particular, an extendible structure has infinitely many automorphisms. In this note we give an example of an extendible countable structure ($M_{1}$ below) with a rigid elementary extension ($M_{2}$).

A set $\cA \subseteq \cP(\omega)$ is said to be \emph{independent} if for all disjoint finite $\cF_{0}, \cF_{1} \subseteq \cA$,
\[\bigcap \cF_{0} \cap \bigcap\{ \omega \setminus A : A \in \cF_{1}\}\] is infinite. A natural recursive argument builds a tree $T \subseteq 2^{<\omega}$ whose infinite paths correspond to an independent family of cardinality $2^{\aleph_{0}}$. Splitting this family into countably many pieces of cardinality $2^{\aleph_{0}}$, and then modifying the members of each piece accordingly, one can produce an independent family $\cA$ with the additional property that for all finite disjoint $u, v \subseteq \omega$ there are $2^{\aleph_{0}}$ many $A \in \cA$ with
$A \cap (u \cup v) = u$. Let call such a family an \emph{improved independent} set.

\begin{df}\label{gisdef} We define a \emph{good independent sequence} to be a set
\[\bar{A} = \langle A_{\alpha, n} : \alpha \leq 2^{\aleph_{0}}, n < \omega \rangle\] such that
\begin{enumerate}
\item $\{ A_{\alpha, n} : \alpha \leq 2^{\aleph_{0}}, n < \omega \}$ is an independent family;
\item the function $(\alpha, n) \mapsto A_{\alpha, n}$ is injective;
\item for all $k, m,n < \omega$ and all $u_{\ell} \subseteq m$ $(\ell < n)$, there exist $2^{\aleph_{0}}$ many ordinals $\alpha$ such that
\begin{enumerate}
\item $\alpha = k \mod \omega$,
\item\label{goodc} $\forall \ell < n$ $A_{\alpha, \ell} \cap m = u_{\ell}$,
\end{enumerate}
\item\label{dshake} for all $i < j < \omega$ there exist an  $n \in \omega$ such that $|A_{2^{\aleph_{0}}, n} \cap \{ i,j\}| = 1$.
\end{enumerate}
\end{df}

Given an improved independent set $\cA$, one can easily find a good independent sequence whose members are in $\cA$.  We let $\bar{A}$ be one such family.



We let $\tau$ be the vocabulary consisting of unary predicates $P$ and $Q$, and binary predicates $R_{n}$ $(n \in \omega)$.
We let $N_{2}$ be the following $\tau$-structure.

\begin{itemize}
\item  The set of elements is of $N$ is
$\{ b_{i} : i < \omega \} \cup \{ c_{\alpha} : \alpha \leq 2^{\aleph_{0}}\}$,  with no repetition.
\item $Q^{N_{2}} = \{ b_{i} : i < \omega \}$
\item $P^{N_{2}} = \{ c_{\alpha} : \alpha \leq 2^{\aleph_{0}}\}$
\item $R_{n}^{N_{2}} \subseteq Q^{N_{2}} \times P^{N_{2}}$ is defined as follows:
\begin{itemize}
\item $(b_{i}, c_{2^{\aleph_{0}}}) \in R_{n}^{N_{2}}$ if and only if $i \in A_{2^{\aleph_{0}},n}$;
\item if $\alpha < 2^{\aleph_{0}}$ and $\alpha = m \mod \omega$, then $(b_{i}, c_{\alpha}) \in R_{n}^{N_{2}}$ if and only if $n \leq m$ and $i \in A_{\alpha, n}$.
\end{itemize}
\end{itemize}
Let $N_{1}$ be the restriction of $N_{2}$ to
\[\{ b_{i} : i < \omega \} \cup \{ c_{\alpha} : \alpha < 2^{\aleph_{0}}\}.\]

Let $\cX\prec (H((2^{\aleph_{0}})^{+}),\in)$ be countable, with $\bar{A} \in \mathcal{B}$.
We let
\begin{itemize}
\item $M_{1}$ be the restriction of $N_{2}$ to
$\{ b_{i} : i < \omega \} \cup \{ c_{\alpha} : \alpha \in 2^{\aleph_{0}} \cap \cX\}$;
\item $M_{2}$ be the restriction of $N_{2}$ to
$\{ b_{i} : i < \omega \} \cup \{ c_{\alpha} : \alpha \in (2^{\aleph_{0}} + 1) \cap \cX\}$.
\end{itemize}


For each $m \in \omega$, let $\tau_{m} = \{ P, Q, R_{n} : n \leq m\}$.
Proposition \ref{claimone} shows that $M_{1}$ is extendible, and that $M_{1}$ is elementary in $M_{2}$. Proposition \ref{claim2} shows that $M_{2}$ is
rigid.

\begin{prop}\label{claimone} The following elementarity relations hold.
\begin{enumerate}
\item\label{firstrel} $N_{1} \prec N_{2}$.
\item\label{secondrel} $M_{1} \prec M_{2}$.
\item\label{thirdrel} $M_{1} \prec_{\infty, \aleph_{0}} N_{1}$
\end{enumerate}
\end{prop}

\begin{proof}
For part (\ref{firstrel}), it suffices to prove that for every $m \in \omega$, $N_{1} \restrict \tau_{m} \prec N_{2} \restrict \tau_{m}$.
In fact we will show that $N_{1} \restrict \tau_{m} \prec_{\infty, \aleph_{0}} N_{2} \restrict \tau_{m}$ by showing that player  $II$ has a winning strategy in the back-forth-game of length $\omega$ between these two structures,
below the play pairing $c_{m}$ with $c_{2^{\aleph_{0}}}$. By Karp's Theorem (Corollary 3.5.3 of \cite{Hodges}, for instance), this suffices.
To show that $II$ has such a strategy, let $F$ be the set of relation-preserving maps $f$ from \[\{ b_{i} : i < \omega\} \cup \{ c_{\alpha} : \alpha < 2^{\aleph_{0}}\}\]
to \[\{ b_{i} : i < \omega\} \cup \{ c_{\alpha} : \alpha \leq 2^{\aleph_{0}}\}\] with $f(c_{m}) = c_{2^{\aleph_{0}}}$ and $\alpha = \beta \mod \omega$ whenever $\alpha \neq k$ and $f(c_{\alpha}) = c_{\beta}$. It suffices then to show that for each $f$, each $i \in \omega$ and each $\alpha < 2^{\aleph_{0}}$, there exists an $f' \in F$ containing $f$ with $b_{i}$ and $c_{\alpha}$ in both the domain and the range of $f'$.
The independence of $\cA$ implies that elements of the form $b_{i}$ can be added. Item (\ref{goodc}) of Definition \ref{gisdef} implies that elements of the form $c_{\alpha}$ can be dealt with.

Part (\ref{secondrel}) follows from part (\ref{firstrel}) and the elementarity of $\cX$ in $H((2^{\aleph_{0}})^{+})$.




Part (\ref{thirdrel}) follows from the fact that player  $II$ has a winning strategy in the back-and-forth game of length $\omega$ between $M_{1}$ and $N_{1}$ (again using Karp's Theorem). For this, let $F$ be the set of relation-preserving maps $f$ from
\[\{ b_{i} : i < \omega\} \cup \{ c_{\alpha} : \alpha  \in 2^{\aleph_{0}} \cap \cX\}\]
to
\[\{ b_{i} : i < \omega\} \cup \{ c_{\alpha} : \alpha < 2^{\aleph_{0}}\}\] with $\alpha = \beta \mod \omega$ whenever $f(c_{\alpha}) = c_{\beta}$. It suffices then to show that $II$ has a response to any move by player $I$ meeting this condition. This follows just as in the proof of part (\ref{firstrel}).
\end{proof}

\begin{prop}\label{claim2} Every substructure of $N_{2}$ containing $M_{2}$ is rigid.
\end{prop}

\begin{proof} Let $M$ a substructure of $N_{2}$ containing $M_{2}$ and let $\pi$ be an automorphism of $M$. Since $c_{2^{\aleph_{0}}}$ is the only $c \in P^{M}$ such that for each $n \in \omega$ there is a $b \in Q^{M}$ wth $(b,c) \in R^{M}_{n}$, $\pi(c_{2^{\aleph_{0}}}) = c_{2^{\aleph_{0}}}$. Now $\pi[Q^{M_{2}}] = Q^{M_{2}}$, and if $i < j < \omega$ then, by condition (\ref{dshake}) of Definition \ref{gisdef}, for some $n \in \omega$, \[(b_{i}, c_{2^{\aleph_{0}}}) \in R^{M}_{n} \leftrightarrow (b_{j}, c_{2^{\aleph_{0}}}) \not\in R^{M}_{n},\] so $\pi$ is the identity function on $Q^{M_{2}}$.

Finally, for each $c_{\alpha}$ in $P^{M}$, and for each $i \in \omega$, $(b_{i}, c_{\alpha}) \in R^{M}_{0}$ if and only if $i \in A_{\alpha, 0}$.
Since $\pi$ fixes $\{ b_{i} : i \in \omega\}$ pointwise, and the sets $A_{\alpha, 0}$ $(\alpha < 2^{\aleph_{0}})$ are distinct, $\pi$ must be the identity function on $P^{M}$.
\end{proof}

Say that a set $Z \subseteq (2^{\aleph_{0}} + 1)$ is \emph{robust} if for each $k \in \omega$ there are infinitely many $\alpha \in Z$ with $\alpha = k \mod \omega$. For each $Z \subseteq (2^{\aleph_{0}} + 1)$, let $M_{Z}$ be the restriction of $M_{1}$ to $\{ b_{i} : i \in \omega \} \cup
\{ c_{\alpha} : \alpha \in Z\}$. The arguments above show the following facts, which imply that there are, up to isomorphism, continuum many countable models elementarily equivalent to $M_{1}$.

\begin{itemize}
\item If $Z \subseteq (2^{\alpha_{0}} + 1)$ is robust, then $M_{Z}$ is elementarily equivalent to $M_{1}$.
\item If $Z_{1}, Z_{2}$ are robust subsets of $(2^{\aleph_{0}} + 1)$ with $2^{\aleph_{0}} \in Z_{1} \cap Z_{2}$, then either $Z_{1} = Z_{2}$ or
$M_{Z_{1}}$ and $M_{Z_{2}}$ are nonisomorphic.
\end{itemize}

One can also modify the example above to make $N_{1}$ and $N_{2}$ arbitrarily large, simply by taking as many disjoint copies of $N_{1}$ or $N_{2}$ as desired.

\begin{ques}
  Can an extendible countable model have at least one but only countably many rigid elementary extensions?
\end{ques}


\begin{thebibliography}{99}

\bibitem{Gao}
S. Gao, \emph{On automorphism groups of countable structures},
   J. Symbolic Logic 63 (3), 1998, 891--896


\bibitem{Hodges}
W. Hodges, {\bf Model theory}, Encyclopedia of Mathematics and its Applications, 42. Cambridge University Press, Cambridge, 1993


\end{thebibliography}
\end{document}